\documentclass[final,leqno]{siamltex}
\usepackage{graphicx}
\usepackage[T1]{fontenc}
\usepackage{psfrag}
\usepackage{amsmath}
\usepackage{amssymb}
\usepackage{color}
\usepackage{times}
\usepackage{afterpage}
\usepackage{mathtools}
\usepackage{bbm}
\newtheorem{remark}{Remark}[section]

\DeclarePairedDelimiter{\ceil}{\lceil}{\rceil}

\usepackage[letterpaper=true,colorlinks=true,
linkcolor=red,filecolor=green,citecolor=blue,
pdfpagemode=None]{hyperref}

\usepackage{xcolor}
\colorlet{linkequation}{cyan}
\newcommand*{\SavedEqref}{}
\let\SavedEqref\eqref
\renewcommand*{\eqref}[1]{%
  \begingroup
    \hypersetup{
      linkcolor=linkequation,
      linkbordercolor=linkequation,
    }%
    \SavedEqref{#1}%
  \endgroup
}

\makeatletter
\newcommand{\clonelabel}[2]{\@bsphack
  \expandafter\ifx\csname r@#2\endcsname\relax
  \else\protected@write\@auxout{}{\string\newlabel{#1}%
    {\csname r@#2\endcsname}}%
  \fi
  \expandafter\ifx\csname r@#2@cref\endcsname\relax
  \else\protected@write\@auxout{}{\string\newlabel{#1@cref}%
    {\csname r@#2@cref\endcsname}}%
  \fi
  \@esphack}
\makeatother
\begin{document}

\title{Dirichlet-Neumann Waveform Relaxation
  Algorithm for Time Fractional Diffusion Equation in Heterogeneous Media}

\author{Soura Sana\footnotemark[1]\and Bankim C Mandal\footnotemark[2]}

\maketitle

\begin{abstract}
 In this article, we have studied the convergence behavior of the Dirichlet-Neumann waveform relaxation algorithms for time-fractional sub-diffusion and diffusion wave equations in 1D \& 2D for regular domains, where the dimensionless diffusion coefficient takes different constant values in different subdomains. From numerical experiments, we first capture the optimal relaxation parameters. Using these optimal relaxation parameters, our analysis estimates the rate of change of the convergence behavior against the fractional order and time. We have performed our analysis in multiple subdomain cases for both in 1D \& 2D.

\end{abstract}



\renewcommand{\thefootnote}{\fnsymbol{footnote}}
\footnotetext[1]{School of Basic Sciences, IIT Bhubaneswar, India ({\tt ss87@iitbbs.ac.in}).}
\footnotetext[2]{School of Basic Sciences, IIT Bhubaneswar, India ({\tt bmandal@iitbbs.ac.in}).}

\section{Model problem}
We aim to extend the DNWR algorithm for time-fractional diffusion equations. The model problem we have taken is linear, but the time derivative is of fractional order $2\nu$, where $\nu$ can take any values from zero to one. Therefore, our model incorporates three fundamental basic PDEs: elliptic, parabolic, and hyperbolic. More details about the model and its application can be found in \cite{evangelista2018fractional,lenzi2016anomalous}. To define the sub-diffusion and diffusion-wave, first, we consider the domain $\Omega$ in $\mathbb{R}^d$, and then we have:
\begin{equation}\label{model_problem}
\begin{cases}
\frac{\partial^{2\nu}}{\partial t^{2\nu}}u = \nabla\cdot\left(\kappa(\boldsymbol{x},t)\nabla u\right)+f(\boldsymbol{x},t), & \textrm{in}\; \Omega\times(0,T),\\
u(\boldsymbol{x},t) = g(\boldsymbol{x},t), & \textrm{on}\; \partial\Omega\times(0,T),\\
u(\boldsymbol{x},0) = u_{0}(\boldsymbol{x}), & \textrm{in}\; \Omega.
\end{cases}
\end{equation}
where $\kappa(\boldsymbol{x},t)>0$ be the dimensionless diffusion coefficient and $\frac{\partial^{\alpha}}{\partial t^{\alpha}}$ is the Caputo fractional derivative defined for order $\alpha$, $n-1<\alpha<n$ and $n \in \mathbb{N}$  as follows:
$$\frac{\partial^{\alpha}}{\partial t^{\alpha}}x(t) := \frac{1}{\Gamma(n-\alpha)}\int_0^t(t-\tau)^{n-\alpha-1}x^{(n)}(\tau)d\tau. $$
For further analysis, we assume that the solution $u(\boldsymbol{x},t)$ and the flux is continuous even for finite jump discontinuous $\kappa(\boldsymbol{x},t)$ in space. In support of the existence and uniqueness of the weak solution of \eqref{model_problem} we refer to \cite{van2021existence}. With this assumption, we next introduce the DNWR algorithm:

\section{The Dirichlet-Neumann Waveform Relaxation algorithm}\label{Section2}
Many forms of DNWR algorithm depend on the placement of Dirichlet-Neumann, Neumann-Neumann, and Dirichlet-Dirichlet boundary conditions already exist in literature; for more details, see \cite{gander2021dirichlet}. Here we have chosen the particular one, which is given in Figure \ref{fig1}. The reason behind this combination is that it can handle long-time processes much better than other ones. So to use this combination in model problem \eqref{model_problem}, we subdivide the domain $\Omega$ into $N$ subdomains $\Omega_{i}$, $1\leq i\leq N$ with no overlaps and no cross points.
\begin{figure} \label{fig1}
  \centering
  \includegraphics[width=1\textwidth]{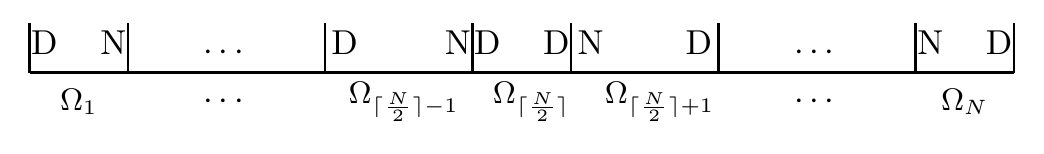}
 \vspace{-2.5em}
  \caption{Particular form of DNWR algorithm.}
\end{figure}
Set original boundary $\Gamma_i:=\partial\Omega\cap\partial\Omega_i, 1\leq i \leq N$ and artificial one as
$\Lambda_i:= \partial\Omega_{i}\cup\partial\Omega_{i+1}$ for $1\leq i \leq N-1$. We choose $\psi_{i}^{(0)}(\boldsymbol{x},t)$ as an initial guess in some appropriate space along the interfaces $\Lambda_{i}\times (0,T)$. Then DNWR algorithm reads:
\begin{equation}\label{DNWRe1}
  \begin{array}{rcll}
  \begin{cases}
   \frac{\partial^{2\nu}}{\partial t^{2\nu}}u_{i}^{(k)} = \nabla\cdot\left(\kappa(\boldsymbol{x},t)\nabla u_{i}^{(k)}\right) + f, & \mbox{in $\Omega_{i}$},\\
    u_{i}^{(k)}(\boldsymbol{x},0) = u_{0}(\boldsymbol{x}), & \mbox{in $\Omega_{i}$},\\
    u_{i}^{(k)} = g, & \mbox{on $\Gamma_i$},
   \end{cases}
  \end{array}
\end{equation}
with the boundary conditions along the artificial interfaces: for $i <\ceil{N/2}$
\begin{equation}
	\begin{cases*}
		u_{i-1}^{(k)} = \psi_{i-1}^{(k-1)},& \mbox{on $\Lambda_{i-1}, i \neq 1$},\\
		\partial_{\boldsymbol{n}_{i,i+1}}u_{i}^{(k)} = -\partial_{\boldsymbol{n}_{i+1,i}}u_{i+1}^{(k)}, & \mbox{on $\Lambda_{i}$}
	\end{cases*}
\end{equation}
for $i =\ceil{N/2}$
\begin{equation}
	\begin{cases*}
		u_{i-1}^{(k)} = \psi_{i-1}^{(k-1)},  & \mbox{on $\Lambda_{i-1}$},\\
		u_{i}^{(k)} = \psi_{i}^{(k-1)}, & \mbox{on $\Lambda_{i}$},
	\end{cases*}
\end{equation}
for $i >\ceil{N/2}$
\begin{equation}
	\begin{cases*}
		\partial_{\boldsymbol{n}_{i,i-1}}u_{i}^{(k)} = -\partial_{\boldsymbol{n}_{i-1,i}}u_{i-1}^{(k)}, & \mbox{on $\Lambda_{i-1}$}\\
	     u_{i}^{(k)} = \psi_{i}^{(k-1)},& \mbox{on $\Lambda_{i}, i \neq N$},
	\end{cases*}
\end{equation}
and then with the relaxation parameter $\theta\in(0,1]$ update the interface data using
\begin{equation}\label{DNWRe2}
	\begin{cases*}
  h_{i}^{(k)}(\boldsymbol{x},t)=\theta u_{i}^{(k)}\left|_{\Lambda_{i}\times(0,T)}\right.+(1-\theta)\psi_{i}^{(k-1)}(\boldsymbol{x},t), \quad 1 \leq i < \ceil{N/2} \\
   h_i^{(k)}(\boldsymbol{x},t)=\theta u_{i+1}^{(k)}\left|_{\Lambda_{i}\times(0,T)}\right.+(1-\theta)\psi_i^{(k-1)}(\boldsymbol{x},t), \quad \ceil{N/2} \leq i \leq N-1.
  	\end{cases*}
\end{equation}
 Our goal is to study the convergence of  $w_{i}^{(k-1)}(\boldsymbol{x},t) := u_i|_{\Gamma_i\times(0,T)} - \psi_{i}^{(k-1)}(\boldsymbol{x},t)$ as $k \to \infty$. 
However, before going to the main theorems on convergence, we need some estimates, which we will discuss in the next section.

\section{Auxiliary Results}\label{Section4}
In this section, we first discuss a few norm inequalities and then some kernel estimates, which are necessary for subsequent sections.
% New lemma 1
\begin{lemma} \label{SimpleLaplaceLemma}
Let $g$ and $w$ be two real-valued functions in $(0,\infty)$ with
$\hat{w}(s)=\mathcal{L}\left\{ w(t)\right\}$ the Laplace transform of
$w$. Then for $t\in(0,T)$, we have the following properties:
\begin{enumerate}
  \item[(i)] \label{L2} $\| g*w\|_{L^{1}(0,t)}\leq\| g\|_{L^{1}(0,t)}\| w\|_{L^{1}(0,t)}.$

  \item[(ii)] \label{L3} $\|g*w\|_{L^{\infty}(0,t)} \leq \| g\|_{L^{\infty}(0,t)} \|w\|_{L^{1}(0,t)}.$
  
 \item[(iii)] \label{L5}If $w(\tau)\geq0$ is $L^1$-integrable in $(0,t)$, then $\|w(.)\|_{L^1(0,t)}\leq \lim_{s \to 0+}\hat{w}(Re(s))$.
\end{enumerate}
\end{lemma}

% New Lemma 2
\begin{lemma} \label{invL_expLemma}
	For $\alpha \in (0,1)$ and $l,t>0$, the following results hold:
	\begin{enumerate}
		\item[(i)]  The inverse Laplace transform of $  e^{-l s^{\alpha}}$ is:
		\begin{equation*} 
			\mathcal{L}^{-1}\left\{e^{-l s^{\alpha}}\right\} 
			=l\alpha t^{-(\alpha+1)}M_{\alpha}\left(lt^{-\alpha}\right),
		\end{equation*} 
		where $M_{\alpha}(x), x\in(0,\infty)$ be the M-Wright function.
		
		\item[(ii)]
		\begin{equation*} 
			\left\|\mathcal{L}^{-1}\left\{e^{-l s^{\alpha}}\right\}\right\|_{L^1(0,t)} \leq \exp\left(-(1-\alpha)\left(\frac{\alpha}{t}\right)^{\alpha/(1-\alpha)}l^{1/(1-\alpha)}\right).
		\end{equation*}  
	\end{enumerate}
\end{lemma}

% New Lemma 3
\begin{lemma} \label{PositivityLemma}
  Let $0 \leq l_1 <l_2$ and $s$ be a complex variable. Then, for $t\in(0,\infty)$ and
  \begin{enumerate}
  	\item [(i)] 
  	for $0<\alpha< 1$
  	\begin{equation*}
  		\Phi(t):=\mathcal{L}^{-1}\left\{ \frac{\sinh(l_1s^{\alpha})}
  		{\sinh(l_2s^{\alpha})}\right\} \quad\mbox{and}\quad
  		\Psi(t):=\mathcal{L}^{-1}\left\{\frac{\cosh(l_1s^{\alpha})}{\cosh(l_2s^{\alpha})}\right\}
  	\end{equation*}
  exist and $L^1$ integrable.
  \item [(ii)]
  for $0<\alpha \leq 1/2$
  \begin{equation*}
  	\Phi(t) \geq 0 \textit{ and } \Psi(t) \geq 0.
  \end{equation*}
  \end{enumerate}
\end{lemma}

\section{Convergence of DNWR Algorithm}\label{sectionDNWR}
We are now in a position to present the main convergence result for the DNWR algorithm \eqref{DNWRe1}-\eqref{DNWRe2}. To have a clear picture of theoretical convergence, we choose the 1D form of the sub-diffusion and diffusion-wave model on $2m+1, m \in \mathbb{N}$ subdomains, but for generality consider the heterogeneous problem with $\kappa(\boldsymbol{x},t)=\kappa_{\mu}$, on $\Omega_{\mu}$ where $1 \leq \mu \leq m$ and $\kappa(\boldsymbol{x},t)=\kappa_{\iota}$, on $\Omega_{\iota}$ where $m+2 \leq \iota \leq 2m+1$ and $\kappa(\boldsymbol{x},t)=\kappa_{m+1}$, on $\Omega_{m+1}$. Considering $w_i^{(0)}, 1 \leq i \leq 2m$ be the initial error term on artificial interfaces, the Laplace transform in time converts the error equations of \eqref{DNWRe1}-\eqref{DNWRe2} into the following sets of ODEs:
\begin{align}\label{DNWR1}
	\begin{cases}
	(s^{2\nu}-\kappa_{m+1}\frac{d^2}{dx^2})\hat{u}_{m+1}^{(k)} = 0, & \textrm{in $\Omega_{m+1}$},\\
	\hat{u}_{m+1}^{(k)} = \hat{w}_{m}^{(k-1)}, & \textrm{at $x_{m}$}\\
	\hat{u}_{m+1}^{(k)} = \hat{w}_{m+1}^{(k)}, & \textrm{at $x_{m+1}$}
     \end{cases}
\end{align}
\begin{equation}\label{DNWR2}
	\begin{array}{rcll}
	\begin{cases}
	  (s^{2\nu}-\kappa_{\mu}\frac{d^2}{dx^2})\hat{u}_{\mu}^{(k)} = 0,& \textrm{in $\Omega_{\mu}$}\\
	  \hat{u}_{\mu}^{(k)} = \hat{w}_{\mu-1}^{(k-1)}, & \textrm{at $x_{\mu-1}$}\\
	  \kappa_{\mu}\frac{d}{dx}\hat{u}_{\mu}^{(k)} = \kappa_{\mu+1}\frac{d}{dx} \hat{u}_{\mu+1}^{(k)}, & \textrm{at $x_{\mu}$}
	  \end{cases}
	\end{array}
	\begin{array}{rcll}
	\begin{cases}
	  (s^{2\nu}-\kappa_{\iota}\frac{d^2}{dx^2})\hat{u}_{\iota}^{(k)} = 0, & \textrm{in $\Omega_{\iota}$}\\
	  \kappa_{\iota}\frac{d}{dx}\hat{u}_{\iota}^{(k)} = \kappa_{\iota-1}\frac{d}{dx} \hat{u}_{\iota-1}^{(k)}, & \textrm{at $x_{\iota-1}$}\\
	  \hat{u}_{\iota}^{(k)} = \hat{w}_{\iota}^{(k-1)}, & \textrm{at $x_{\iota}$}
	  \end{cases}
	\end{array}
\end{equation}
%where 
%\begin{align}
%	\hat{z}_i^{(k)}(s) &= \partial_{x} \hat{u}_{i+1}^{(k)}(x_i,s) \textit{ for } 1 \leq i \leq m, \\
%	\hat{z}_j^{(k)}(s) &= -\partial_{x} \hat{u}_{j}^{(k)}(x_j,s) \textit{ for } m+1 \leq j \leq 2m.
%\end{align}
then update the interface error by
\begin{align} 
	\hat{w}_{\mu}^{(k)}(s)&=\theta_{\mu}\hat{u}_{\mu}^{(k)}(x_{\mu},s)+(1-\theta_{\mu})\hat{w}_{\mu}^{(k-1)}(s),  \textit{ for } 1 \leq \mu \leq m, \label{DNWR3}\\
	\hat{w}_{\iota}^{(k)}(s)&=\theta_{\iota}\hat{u}_{\iota}^{(k)}(x_{\iota},s)+(1-\theta_{\iota})\hat{w}_{\iota}^{(k-1)}(s), \textit{ for } m+1 \leq \iota \leq 2m. \label{DNWR4}
\end{align}
Now we can solve for each $\hat{u}_{i}^{(k)}$ from \eqref{DNWR1} \& \eqref{DNWR2}, and substitute those values in \eqref{DNWR3} \& \eqref{DNWR4} at artificial interfaces give us the following relation:
\begin{equation} \label{proof9}
	\hat{w}_i^{(k)}(s) = \sum_{j=1}^{2m}\hat{\rho}_{i,j} \hat{w}_j^{(k-1)}(s)
\end{equation}
where $\hat{\rho}_{i,j}$, whose calculation is quite cumbersome so given in Appendix \ref{AppendixA}, takes the form for $1 \leq i \leq m$ and $1 \leq j \leq 2m$:
\begin{equation*} 
	\hat{\rho}_{i,j} = 
	\begin{cases}
		1- \theta_i\left(1 + \sqrt{\frac{\kappa_{i+1}}{\kappa_i}}\right) + \theta_i\sqrt{\frac{\kappa_{i+1}}{\kappa_i}} \left(1-\frac{\sigma_i \sigma_{i+1}}{\gamma_i \gamma_{i+1}}\right), &j=i,i<m \\
		1- \theta_m\left(1 + \sqrt{\frac{\kappa_{m+1}}{\kappa_m}}\right) + \theta_m\sqrt{\frac{\kappa_{m+1}}{\kappa_m}} \left(1-\frac{\sigma_m \gamma_{m+1}}{\gamma_m \sigma{m+1}}\right), &j=i=m, \\
		\frac{\theta_i}{\gamma_i}, &j = i-1, i>1, \\
		\theta_i \sqrt{\frac{\kappa_{j+1}}{\kappa_i}} \frac{-\sigma_i\sigma_{j+1}}{\gamma_i\gamma_{i+1}\cdots\gamma_j}, &i+1\leq j \leq m-1, \\
		\theta_i \sqrt{\frac{\kappa_{m+1}}{\kappa_i}} \frac{-\gamma_{m+1}\sigma_{i}}{\gamma_i\gamma_{i+1}\cdots\gamma_m \sigma_{m+1}}, &j = m, i<m, \\
		\theta_i \sqrt{\frac{\kappa_{m+1}}{\kappa_i}} \frac{\sigma_{i}}{\gamma_i\gamma_{i+1}\cdots\gamma_m \sigma_{m+1}}, &j = m+1,\\	
		0, &\textit{ otherwise. }
	\end{cases}
\end{equation*}
Due to the particular symmetric structure of the $\hat{\rho}$ matrix, the lower half of it, that is, $\hat{\rho}_{i,j}$ such that $m+1 \leq i \leq 2m$ for all $j$  can be easily deduced from the upper half by just changing all indices $i_1$ by $2m+2-i_1$; except $\hat{\rho}_{i,j},\theta_i$, which will be replaced by $2m+1-i_1$, where $1\leq i_1 \leq m$.
For having super-linear convergence, we choose $\theta_i = 1/(1+\sqrt{\kappa_{i+1}/\kappa_{i}}) =: \theta^*_i$ for $1\leq i \leq m$ and get:
\begin{equation} \label{dij}
	\hat{\rho}_{i,j} = 
	\begin{cases}
		\theta^*_i\sqrt{\frac{\kappa_{i+1}}{\kappa_i}} \left(1-\frac{\sigma_i \sigma_{i+1}}{\gamma_i \gamma_{i+1}}\right), &j=i,i<m \\
		\theta^*_m\sqrt{\frac{\kappa_{m+1}}{\kappa_m}} \left(1-\frac{\sigma_m \gamma_{m+1}}{\gamma_m \sigma{m+1}}\right), &j=i=m, \\
		\frac{\theta^*_i}{\gamma_i}, & j = i-1, i>1, \\
		\theta^*_i \sqrt{\frac{\kappa_{j+1}}{\kappa_i}} \frac{-\sigma_i\sigma_{j+1}}{\gamma_i\gamma_{i+1}\cdots\gamma_j}, &i+1\leq j \leq m-1, \\
		\theta^*_i \sqrt{\frac{\kappa_{m+1}}{\kappa_i}} \frac{-\gamma_{m+1}\sigma_{i}}{\gamma_i\gamma_{i+1}\cdots\gamma_m \sigma_{m+1}}, & j = m, i<m, \\
		\theta^*_i \sqrt{\frac{\kappa_{m+1}}{\kappa_i}} \frac{\sigma_{i}}{\gamma_i\gamma_{i+1}\cdots\gamma_m \sigma_{m+1}}, &j = m+1,\\	
		0, & \textit{ otherwise. }
	\end{cases}
\end{equation}
%And for $m+1 \leq i \leq 2m$ choose $\theta_i = 1/(1+\sqrt{\kappa_{i}/\kappa_{i+1}}) =: \theta^*_i$ obtain:
%\begin{equation} \label{dij2}
%	\hat{\rho}_{i,j} = 
%	\begin{cases}
%		\theta^*_i\sqrt{\frac{\kappa_{i+1}}{\kappa_i}} \left(1-\frac{\sigma_i \sigma_{i+1}}{\gamma_i \gamma_{i+1}}\right), &j=i,i<m \\
%		\theta^*_m\sqrt{\frac{\kappa_{m+1}}{\kappa_m}} \left(1-\frac{\sigma_m \gamma_{m+1}}{\gamma_m \sigma{m+1}}\right), &j=i=m \\
%		\frac{\theta^*_i}{\gamma_i}, & j = i-1, i>1, \\
%		\theta^*_i \sqrt{\frac{\kappa_{j+1}}{\kappa_i}} \frac{-\sigma_i\sigma_{j+1}}{\gamma_i\gamma_{i+1}\cdots\gamma_j}, &i+1\leq j \leq m-1, \\
%		\theta^*_i \sqrt{\frac{\kappa_{m+1}}{\kappa_i}} \frac{-\gamma_{m+1}\sigma_{i}}{\gamma_i\gamma_{i+1}\cdots\gamma_m \sigma_{m+1}}, & j = m, i<m, \\
%		\theta^*_i \sqrt{\frac{\kappa_{m+1}}{\kappa_i}} \frac{\sigma_{i}}{\gamma_i\gamma_{i+1}\cdots\gamma_m \sigma_{m+1}}, &j = m+1,\\	
%		0, & \textit{ otherwise }
%	\end{cases}
%\end{equation}

%\begin{lemma} \label{estimate_sinh_sinh}
%	For $l_2 > l_1 \geq 0$; $\mathcal{L}^{-1}\left\{\frac{\sinh(l_1s^{\alpha})}{\sinh(l_2s^{\alpha})}\right\}, \mathcal{L}^{-1}\left\{\frac{\cosh(l_1s^{\alpha})}{\cosh(l_2s^{\alpha})}\right\} \in L^1 (0,t)$.
%\end{lemma}	
%\begin{proof}
%	From Lemma in , we have $\mathcal{L}^{-1}\left\{\frac{\sinh(l_1s^{\alpha})}{\sinh(l_2s^{\alpha})}\right\} \in L^1 (0,t)$.
%	The proof for $\mathcal{L}^{-1}\left\{\frac{\cosh(l_1s^{\alpha})}{\cosh(l_2s^{\alpha})}\right\} \in L^1 (0,t)$ is similar, hence omitted.
%\end{proof}	

\begin{lemma} \label{lemma_2}
	Suppose $l_1,l_2,\cdots,l_n$ are positive real numbers with $l = \min_i l_i$ and $0 < \alpha \leq 1/2$. Let $\hat{P}_i(s)$ be the Laplace transform  of $P_i(t)$. Then we have the followings bounds:
	
	\begin{enumerate}		
		\item [(i)] $\|\mathcal{L}^{-1}\left\{\hat{P}_1(s)\right\}\|_{L^{1}(0,t)} = \|\mathcal{L}^{-1}\left\{\frac{\cosh(l s^{\alpha})\cosh((l_1-l_2)s^{\alpha})}{\cosh(l_1 s^{\alpha})\cosh(l_2 s^{\alpha})}\right\}\|_{L^{1}(0,t)}\leq 1,$
		
		\item [(ii)]  $\|\mathcal{L}^{-1}\{\hat{P}_2(s)\}\|_{L^{1}(0,t)} =  \|\mathcal{L}^{-1}\left\{\frac{\cosh(l s^{\alpha})\sinh(l_1 s^{\alpha})\sinh(l_n s^{\alpha})}{\cosh(l_1 s^{\alpha})\cosh(l_2 s^{\alpha})\cdots \cosh(l_n s^{\alpha})}\right\}\|_{L^{1}(0,t)}\leq 2,$
		
		\item [(iii)] $\|\mathcal{L}^{-1}\{\hat{P}_3(s)\}\|_{L^{1}(0,t)} = \|\mathcal{L}^{-1}\left\{\frac{\cosh(l s^{\alpha})\sinh(l_1 s^{\alpha})\cosh(l_n s^{\alpha})}{\cosh(l_1 s^{\alpha})\cdots\cosh(l_{n-1} s^{\alpha})\sinh(l_n s^{\alpha})}\right\}\|_{L^{1}(0,t)}\leq 1+\left|\frac{l_1-l_n}{l_n}\right|,$
		
		\item [(iv)] $\|\mathcal{L}^{-1}\{\hat{P}_4(s)\}\|_{L^{1}(0,t)} = \|\mathcal{L}^{-1}\left\{\frac{\cosh(l s^{\alpha})\sinh(l_1 s^{\alpha})}{\cosh(l_1 s^{\alpha})\cdots\cosh(l_{n-1} s^{\alpha})\sinh(l_n s^{\alpha})}\right\}\|_{L^{1}(0,t)}\leq \frac{l_1}{l_n}.$
	\end{enumerate}
\end{lemma}

\begin{proof}
	\begin{enumerate}
		\item [(i)] 
		If $l_2 \geq l_1$, from part (i) and (ii) of Lemma \ref{PositivityLemma}, guarantee that the inverse Laplace transform of  $\hat{P}_{11}(s):=\frac{\cosh(l s^{\alpha})}{\cosh(l_1 s^{\alpha})}, \hat{P}_{12}(s) :=\frac{\cosh((l_1-l_2)s^{\alpha})}{\cosh(l_2 s^{\alpha})}$ exist, $L^1$ integrable and non-negative. Therefore using part (i) of Lemma \ref{SimpleLaplaceLemma} we have the following inequality:
		\begin{align*}
			\|P_1(.)\|_{L^1(0,t)} \leq \left\|\mathcal{L}^{-1}\left\{\hat{P}_{11}(s)\right\}\right\|_{L^1(0,t)} \left\|\mathcal{L}^{-1}\left\{\hat{P}_{12}(s)\right\}\right\|_{L^1(0,t)}.
		\end{align*}
		 Now using part (iii) of Lemma \ref{SimpleLaplaceLemma}, we have:  
		\begin{align*}
			\|P_1(.)\|_{L^{1}(0,t)} &  \leq \left\|\mathcal{L}^{-1}\left\{\hat{P}_{11}(s)\right\}\right\|_{L^1(0,t)} \left\|\mathcal{L}^{-1}\left\{\hat{P}_{12}(s)\right\}\right\|_{L^1(0,t)} \\
			&\leq \lim_{s \to 0} \hat{P}_{11}(Re(s)) \hat{P}_{12}(Re(s))= 1.
		\end{align*}
	    If $l_2 < l_1$, then we choose $\hat{P}_{11}(s):=\frac{\cosh(l s^{\alpha})}{\cosh(l_2 s^{\alpha})}, \hat{P}_{12}(s) :=\frac{\cosh((l_1-l_2)s^{\alpha})}{\cosh(l_1 s^{\alpha})}$ and follow the same procedure as earlier. Hence combining two cases give us the result:
	    \begin{equation*}
	    	\|\mathcal{L}^{-1}\left\{\hat{P}_1(s)\right\}\|_{L^{1}(0,t)} \leq 1.
	    \end{equation*}
		\item[(ii)]
		To have the $L^1$ estimate on $P_2(t)$, we reduce $\hat{P}_2 (s) = \hat{P}_{21} (s) - \hat{P}_{22} (s)$, where $ \hat{P}_{21} (s) := \frac{\cosh(l s^{\alpha})}{\cosh(l_2 s^{\alpha})\cdots \cosh(l_{n-1} s^{\alpha})}$ and $\hat{P}_{22}(s) := \frac{\cosh(l s^{\alpha})\cosh((l_1 - l_n) s^{\alpha})}{\cosh(l_1 s^{\alpha})\cosh(l_2 s^{\alpha})\cdots \cosh(l_n s^{\alpha})}$. Now we use the part(i) of Lemma \ref{SimpleLaplaceLemma} and have:
		\begin{align*}
			\|\mathcal{L}^{-1}\left\{\hat{P}_{21}(s)\right\}\|_{L^{1}(0,t)} \leq \left\|\mathcal{L}^{-1}\left\{\frac{\cosh(l s^{\alpha})}{\cosh(l_2 s^{\alpha})}\right\}\right\|_{L^1(0,t)} \cdots \left\|\mathcal{L}^{-1}\left\{\frac{1}{\cosh(l_{n-1} s^{\alpha})}\right\}\right\|_{L^1(0,t)}
		\end{align*}
	    Again using the same kind of reasoning as part (i), we have:
	     \begin{align} \label{sub1}
	    	\|\mathcal{L}^{-1}\left\{\hat{P}_{21}(s)\right\}\|_{L^{1}(0,t)} \leq 1.
	    \end{align}
        Proof of the estimates for $\hat{P}_{22}(s)$ for the cases $l_1 \leq l_n$ and $l_1> l_n$ can be done separately and are obvious.
        Therefore 
        \begin{equation}\label{sub2}
        	\|\mathcal{L}^{-1}\left\{\hat{P}_{22}(s)\right\}\|_{L^{1}(0,t)} \leq 1.
        \end{equation}
        Hence, Adding \eqref{sub1} \& \eqref{sub2}, we obtain:
         $$\|\mathcal{L}^{-1}\{\hat{P}_2(s)\}\|_{L^{1}(0,t)}   \leq 2.$$
	
		\item [(iii)]
		We split $\hat{P}_3 (s) = \hat{P}_{31} (s) + \hat{P}_{32} (s)$ where, $\hat{P}_{31} (s) := \frac{\cosh(l s^{\alpha})}{\cosh(l_2 s^{\alpha})\cdots \cosh(l_{n-1} s^{\alpha})}$, so by using same technique as part (ii) , we have  $\|\mathcal{L}^{-1}\{\hat{P}_{31}(s)\}\|_{L^{1}(0,t)}   \leq 1.$ Now $\hat{P}_{32} (s) := \hat{P}_{31} (s) \hat{D}_{1} (s)$, where $\hat{D}_{1} (s) := \frac{\sinh((l_1-l_n)s^{\alpha})}{\cosh(l_1s^{\alpha})\sinh(l_n s^{\alpha})}$. From Theorem 2 in \ref{mandal2014time}, we have $\|\mathcal{L}^{-1}\{\hat{D}_1(s)\}\|_{L^{1}(0,t)}   \leq \frac{|l_1-l_n|}{l_n}$. Therefore adding both cases give:
		  $$\|\mathcal{L}^{-1}\{\hat{P}_3(s)\}\|_{L^{1}(0,t)} \leq 1+\left|\frac{l_1-l_n}{l_n}\right|.$$
		\item [(iv)]
		Rearrange $ \hat{P}_4 (s) = \hat{P}_{41} (s) + \hat{P}_{42} (s)$ where, $\hat{P}_{41} (s) := \frac{\sinh(l s^{\alpha})}{\cosh(l_2 s^{\alpha})\cdots \cosh(l_{n-1} s^{\alpha})\sinh(l_n s^{\alpha})}$. Therefore from part (i) of Lemma \ref{PositivityLemma} we have
		\begin{align*}
			&\|\mathcal{L}^{-1}\{\hat{P}_{41}(s)\}\|_{L^{1}(0,t)} \\
			&\leq \left\|\mathcal{L}^{-1}\left\{\frac{1}{\cosh(l_2 s^{\alpha})\cdots \cosh(l_{n-1} s^{\alpha})}\right\}\right\|_{L^{1}(0,t)} \left\|\mathcal{L}^{-1}\left\{\frac{\sinh(l s^{\alpha})}{\sinh(l_n s^{\alpha})}\right\}\right\|_{L^{1}(0,t)}
			 \leq \frac{l}{l_n}
		\end{align*}
	    Define $\hat{D}_{2} (s) := \frac{\sinh((l_1-l)s^{\alpha})}{\cosh(l_1s^{\alpha})\sinh(l_n s^{\alpha})}$, then $\hat{P}_{42} (s) := \hat{P}_{41} (s) \hat{D}_{2} (s)$. The proof of $\|\mathcal{L}^{-1}\{\hat{P}_{42}(s)\}\|_{L^{1}(0,t)} \leq \|\mathcal{L}^{-1}\{\hat{P}_{41}(s)\hat{D}_{2}(s)\}\|_{L^{1}(0,t)} \leq \frac{l_1-l}{l_n} $ is same as part (iii). Hence adding both parts follows
	    $\|\mathcal{L}^{-1}\{\hat{P}_4(s)\}\|_{L^{1}(0,t)} \leq \frac{l_1}{l_n}.$ 
	\end{enumerate}
\hfill\end{proof}	

\begin{theorem}[Convergence of DNWR when $0< \nu \leq 1/2$]
		For fixed $T>0$, $0 < \nu \leq 1/2$ and $\theta_i = 1/(1+\sqrt{\kappa_{i+1}/\kappa_i}), 1 \leq i \leq m$ and $\theta_i = 1/(1+\sqrt{\kappa_{i}/\kappa_{i+1}}), m+1 \leq i \leq 2m$ , the DNWR algorithm \eqref{DNWR1}-\eqref{DNWR4} for $2m+1( m \in \mathbb{N})$ subdomains converge super-linearly with the estimate:
			\begin{align*}
			\max_{1\leq i \leq 2m} \|w_i^{(k)}\|_{L^{\infty}(0,T)} &\leq (2c)^k \exp(-A k^{1/(1-\nu)})\max_{1\leq j \leq 2m} \|w_j^{(0)}\|_{L^{\infty}(0,T)},
		\end{align*} 
		where,  $h_j = a_j/\sqrt{\kappa_j}, \, h = \min_{1\leq j \leq 2m+1}a_j,\, A = (1-\nu)\nu^{\nu/(1-\nu)}(h/T^{\nu})^{1/(1-\nu)}$, and $c = \max_i c_i$, where $c_i$ are ginen in \eqref{proof4} \& \eqref{proof8}.
\end{theorem}

\begin{proof}
   Solving the ODEs \eqref{DNWR1}-\eqref{DNWR4} in terms of errors on artificial boundaries give us:
      \begin{equation}\label{proof1}
      	\hat{w}_i^{(k)}(s) = \sum_{j=1}^{2m}\hat{\rho}_{i,j} \hat{w}_j^{(k-1)}(s),
      \end{equation}
  where $\hat{\rho}_{i,j}$ takes the form from \eqref{dij} for $1\leq i \leq m$. Cumulative estimate of $\hat{\rho}_{i,j}^k$ at $k$-th iteration level will be cumbersome, so we estimate some form of $\hat{\rho}_{i,j}$ at each iteration. But before that define $\gamma = \cosh(hs^{\nu})$, where $h = \min_j h_j$. setting $\hat{v}^{(k)}_i(s) := \gamma^k \hat{w}^{(k)}_i(s)$. Which transform \eqref{proof1} into:
   \begin{equation*}
   	\hat{v}_i^{(k)} = \sum_{j=1}^{2m} \gamma\hat{\rho}_{i,j} \hat{v}_j^{(k-1)}.
   \end{equation*}	
 To have our final estimate result in time domain we take inverse Laplace transform on both side:
	\begin{equation*}
		v_i^{(k)} = \sum_{j=1}^{2m} \mathcal{L}^{-1}\{\gamma\hat{\rho}_{i,j}\}\ast v_j^{(k-1)}.
	\end{equation*}	
Before going to the sup-norm estimate we first assume that $\left\|v_i^{(k)}\right\|_{L^{\infty}(0,T)}$ exist. Our assumption is justified because each $\left\|w_i^{(k)}\right\|_{L^{\infty}(0,T)}$ perfectly well-defined. Therefore taking $L^{\infty}$ norm on both side reduce the above equation to:
	\begin{align*}
		\left\|v_i^{(k)}\right\|_{L^{\infty}(0,T)} 
		&\leq \sum_{j=1}^{2m} \left\|\mathcal{L}^{-1}\{\gamma\hat{\rho}_{i,j}\}\ast v_j^{(k-1)}\right\|_{L^{\infty}(0,T)}
	\end{align*}
Applying part (ii) of Lemma \ref{SimpleLaplaceLemma} give:
 \begin{align} \label{proof2}
 	\left\|v_i^{(k)}\right\|_{L^{\infty}(0,T)} 
 	&\leq \sum_{j=1}^{2m} \left\|\mathcal{L}^{-1}\{\gamma\hat{\rho}_{i,j}\}\right\|_{L^{1}(0,T)}\left\|v_j^{(k-1)}\right\|_{L^{\infty}(0,T)} \nonumber\\
 	& \leq \max_j \left\|v_j^{(k-1)}\right\|_{L^{\infty}(0,T)} \sum_{j=1}^{2m} \left\|\mathcal{L}^{-1}\{\gamma\hat{\rho}_{i,j}\}\right\|_{L^{1}(0,T)}.
 \end{align}
Now we use the Lemma \ref{lemma_2} to estimate the summation part of the above inequality. For $i=j, j<m$
\begin{equation*}
	\mathcal{L}^{-1}\{\gamma\hat{\rho}_{i,j}\} = \theta_i^* \sqrt{\frac{\kappa_{i+1}}{\kappa_i}} \mathcal{L}^{-1}\left\{\gamma\left(1-\frac{\sigma_i \sigma_{i+1}}{\gamma_i\gamma_{i+1}}\right)\right\},
\end{equation*}
so from part(i) of Lemma \ref{lemma_2} give us
  \begin{equation*}
  \left\| \mathcal{L}^{-1}\left\{\gamma\left(1-\frac{\sigma_i \sigma_{i+1}}{\gamma_i\gamma_{i+1}}\right)\right\}\right\|_{L^{1}(0,T)} \leq 1.
\end{equation*}
Hence $\left\|\mathcal{L}^{-1}\{\gamma\hat{\rho}_{i,j}\}\right\|_{L^{1}(0,T)} \leq \theta_i^* \sqrt{\frac{\kappa_{i+1}}{\kappa_i}}.$
In a similar manner we show $\left\|\mathcal{L}^{-1}\{\gamma\hat{\rho}_{m,m}\}\right\|_{L^{1}(0,T)} \leq \theta_m^* \sqrt{\frac{\kappa_{m+1}}{\kappa_m}}.$ and  $\left\|\mathcal{L}^{-1}\{\gamma\hat{\rho}_{i,j}\}\right\|_{L^{1}(0,T)} \leq \theta_i^*$ for $j = i-1, i>1$. From part(ii) of Lemma \ref{lemma_2},  we get the estimate $\left\|\mathcal{L}^{-1}\{\gamma\hat{\rho}_{i,j}\}\right\|_{L^{1}(0,T)} \leq 2\theta_i^* \sqrt{\frac{\kappa_{j+1}}{\kappa_i}}$ for $i+1 \leq j \leq m-1$. And for $j=m, i<m$ from part (iii) give:
\begin{equation*}
	\left\|\mathcal{L}^{-1}\{\gamma\hat{\rho}_{i,j}\}\right\|_{L^{1}(0,T)} \leq \theta_i^* \sqrt{\frac{\kappa_{m+1}}{\kappa_i}}\left(1+\left|\frac{h_i-h_{m+1}}{h_{m+1}}\right| \right).
\end{equation*}
Finally for $j=m+1, i<m$  using part(iv) obtain $\left\|\mathcal{L}^{-1}\{\gamma\hat{\rho}_{i,j}\}\right\|_{L^{1}(0,T)} \leq \theta_i^* \sqrt{\frac{\kappa_{m+1}}{\kappa_i}}\frac{h_i}{h_{m+1}}.$
Now putting each estimated value in \eqref{proof2} provide:
 \begin{equation} \label{proof3}
	\left\|v_i^{(k)}\right\|_{L^{\infty}(0,T)} 
	 \leq c_i\max_{1\leq j \leq 2m} \left\|v_j^{(k-1)}\right\|_{L^{\infty}(0,T)}, \hspace{20pt} 1 \leq i \leq m,
\end{equation}
where 
\begin{equation} \label{proof4}
	\begin{cases*}
	c_1 = \theta_1^*\left(\sqrt{\frac{\kappa_2}{\kappa_1}}+2\sum_{j=2}^{m-1}\sqrt{\frac{\kappa_{j+1}}{\kappa_1}}+\sqrt{\frac{\kappa_{m+1}}{\kappa_1}}\frac{2h}{h_{m+1}}\right) \\
	c_i = \theta_i^*\left(1+\sqrt{\frac{\kappa_{i+1}}{\kappa_i}}+2\sum_{j=i+1}^{m-1}\sqrt{\frac{\kappa_{j+1}}{\kappa_i}}+\sqrt{\frac{\kappa_{m+1}}{\kappa_i}}\frac{2h}{h_{m+1}}\right), \hspace{20pt} 2 \leq i \leq m.
	\end{cases*}
\end{equation}
In a similar manner we obtain the estimate:
 \begin{equation} \label{proof7}
	\left\|v_{m+i}^{(k)}\right\|_{L^{\infty}(0,T)} 
	\leq c_{m+i}\max_{1\leq j \leq 2m} \left\|v_j^{(k-1)}\right\|_{L^{\infty}(0,T)}, \hspace{20pt} 1 \leq i \leq m,
\end{equation}
where 
\begin{equation} \label{proof8}
	\begin{cases*}
		c_{2m+1-i} = \theta_{2m+1-i} ^*\left(1+\sqrt{\frac{\kappa_{2m+1-i}}{\kappa_{2m+2-i}}}+2\sum_{j=i+1}^{m-1}\sqrt{\frac{\kappa_{2m+1-j}}{\kappa_{2m+2-i}}}+\sqrt{\frac{\kappa_{m+1}}{\kappa_{2m+2-i}}}\frac{2h}{h_{m+1}}\right), \hspace{5pt} 2 \leq i \leq m,\\
		c_{2m} = \theta_{2m}^*\left(\sqrt{\frac{\kappa_{2m}}{\kappa_{2m+1}}}+2\sum_{j=2}^{m-1}\sqrt{\frac{\kappa_{2m+1-j}}{\kappa_{2m+1}}}+\sqrt{\frac{\kappa_{m+1}}{\kappa_{2m+1}}}\frac{2h}{h_{m+1}}\right).
	\end{cases*}
\end{equation}
Therefore combining \eqref{proof3} \& \eqref{proof7} and define $c := \max_{1 \leq i \leq 2m}c_i$ deduce
 \begin{align*}
	\max_{1 \leq i \leq 2m} \left\|v_i^{(k)}\right\|_{L^{\infty}(0,T)} 
	& \leq c\max_{1 \leq j \leq 2m} \left\|v_j^{(k-1)}\right\|_{L^{\infty}(0,T)}.
\end{align*}
Further reduction in terms of the error on initial guesses give:
\begin{align} \label{proof5}
	\max_{1 \leq i \leq 2m} \left\|v_i^{(k)}\right\|_{L^{\infty}(0,T)} 
	& \leq c^k \max_{1 \leq j \leq 2m} \left\|v_j^{(0)}\right\|_{L^{\infty}(0,T)} \nonumber\\
		& = c^k \max_{1 \leq j \leq 2m} \left\|w_j^{(0)}\right\|_{L^{\infty}(0,T)}.
\end{align}
Now the estimate of the error terms in $k^{th}$ iteration is obtained by
\begin{equation*}
	\left\|w_i^{(k)}\right\|_{L^{\infty}(0,T)} \leq  \left\|\mathcal{L}^{-1}\left\{\frac{1}{\cosh^k(hs^{\nu})}\right\}\right\|_{L^{1}(0,T)} \left\|v_i^{(k)}\right\|_{L^{\infty}(0,T)}. 
\end{equation*}
Which using Lemma further reduce to  
\begin{equation} \label{proof6}
	\left\|w_i^{(k)}\right\|_{L^{\infty}(0,T)} \leq  2^k\left\|\mathcal{L}^{-1}\left\{\exp(-khs^{\nu})\right\}\right\|_{L^{1}(0,T)} \left\|v_i^{(k)}\right\|_{L^{\infty}(0,T)}.
\end{equation}
Finally taking the maximum over $i$ in \eqref{proof6} and plugging the values of  $\max_i \left\|v_i^{(k)}\right\|_{L^{\infty}(0,T)}$ from \eqref{proof5} in \eqref{proof6} we have:
\begin{align*}
	\max_{1\leq i \leq 2m} \|w_i^{(k)}\|_{L^{\infty}(0,T)} &\leq (2c)^k \exp(-A k^{1/(1-\nu)})\max_{1\leq j \leq 2m} \|w_j^{(0)}\|_{L^{\infty}(0,T)}.
\end{align*} 
\hfill\end{proof}

%New Lemma 4
\begin{lemma} \label{estimate_1}
%\textbf{Result 4 (Lemma  in [])} 
For $0 < l_1 , l_2, \alpha \in (0,1)$, the following inequality holds: 
\begin{equation*}
	\left\|\mathcal{L}^{-1}\left\{\frac{\exp(-l_1s^{\alpha})}{1-\exp(-l_2s^{\alpha})}\right\}\right\|_{L^1(0,t)} \leq \left[1+\frac{t^{\alpha}\Gamma(2-\alpha)}{l_2\Lambda^{(1-\alpha)}}\right] \exp\left(-\Lambda\left(\frac{l_1}{t^{\alpha}}\right)^{1/(1-\alpha)}\right),
\end{equation*}	
when $\Lambda = (1-\alpha)\alpha^{\alpha/(1-\alpha)}$.
\end{lemma}
% New Lemma
\begin{lemma} \label{lemma_1}
	Suppose $l_1,l_2,\cdots,l_n$ are positive real numbers with $l = \min_i l_i$ and $1/2 < \alpha < 1$. Let $\mathcal{L}^{-1}\{\hat{f}(s)\}$ represents the inverse Laplace transform of function $f(t)$. Then we have the following inequalities:
	
    {\small\begin{enumerate}	
		\item [(i)] 
		\begin{align*}
			&\|\mathcal{L}^{-1}\left\{\frac{\exp(2l s^{\alpha})\cosh((l_1-l_2)s^{\alpha})}{\cosh(l_1 s^{\alpha})\cosh(l_2 s^{\alpha})}\right\}\|_{L^{1}(0,t)} \\ 
			&\leq 2\left[1+\frac{L(t)}{2l_1}\right]\left[1+\frac{L(t)}{2l_2}\right]\left[\exp\left(-2\Lambda\left(\frac{l_1-l}{t^{\nu}}\right)^{\beta}\right)+\exp\left(-2\Lambda\left(\frac{l_2-l}{t^{\nu}}\right)^{\beta}\right)\right],
		\end{align*}
			
		\item [(ii)]
		 \begin{align*}
		  \|\mathcal{L}^{-1}\left\{\frac{\exp(2l s^{\alpha})}{\cosh(l_1 s^{\alpha})}\right\}\|_{L^{1}(0,t)} \leq 2\left[1+\frac{L(t)}{2l_1}\right]\exp\left(-2\Lambda\left(\frac{l_1-2l}{t^{\nu}}\right)^{\beta}\right),
		\end{align*}
		
		\item [(iii)]
		 \begin{align*}
		 &\|\mathcal{L}^{-1}\left\{\frac{\exp(2l s^{\alpha})\sinh(l_1 s^{\alpha})\sinh(l_n s^{\alpha})}{\cosh(l_1 s^{\alpha})\cosh(l_2 s^{\alpha})\cdots \cosh(l_n s^{\alpha})}\right\}\|_{L^{1}(0,t)}\\ 
		 &\leq 2^{n-2}\left[1+\frac{L(t)}{2l_2}\right]\cdots\left[1+\frac{L(t)}{2l_{n-1}}\right] 
		\exp\left(-(n-2)\Lambda\left(\frac{l_2+\cdots+l_{n-1}-2l}{(n-2)t^{\nu}}\right)^{\beta}\right) \\
		&\hspace{5pt}\left[1+2\left[1+\frac{L(t)}{2l_1}\right]\left[1+\frac{L(t)}{2l_n}\right] 
		\left[\exp\left(-2\Lambda\left(\frac{l_1-l}{t^{\nu}}\right)^{\beta}\right)
		+\exp\left(-2\Lambda\left(\frac{l_n-l}{t^{\nu}}\right)^{\beta}\right)\right]\right],
		\end{align*}
		
		\item [(iv)] 
		\begin{align*}
		&\|\mathcal{L}^{-1}\left\{\frac{\exp(2l s^{\alpha})\sinh(l_1 s^{\alpha})\cosh(l_n s^{\alpha})}{\cosh(l_1 s^{\alpha})\cdots\cosh(l_{n-1} s^{\alpha})\sinh(l_n s^{\alpha})}\right\}\|_{L^{1}(0,t)} \\
		&\leq 2^{n-2}\left[1+\frac{L(t)}{2l_2}\right]\cdots\left[1+\frac{L(t)}{2l_{n-1}}\right]
		\exp\left(-(n-2)\Lambda\left(\frac{l_2+\cdots+l_{n-1}-2l}{(n-2)t^{\nu}}\right)^{\beta}\right)\\
		&\left[1+2\left[1+\frac{L(t)}{2l_1}\right]\left[1+\frac{L(t)}{2l_n}\right]\left[\exp\left(-2\Lambda\left(\frac{l_1-l}{t^{\nu}}\right)^{\beta}\right)+\exp\left(-2\Lambda\left(\frac{l_n-l}{t^{\nu}}\right)^{\beta}\right)\right]\right],
		\end{align*}
			
		\item [(v)]
		 \begin{align*}
		 	&\|\mathcal{L}^{-1}\left\{\frac{\exp(2l s^{\alpha})\sinh(l_1 s^{\alpha})}{\cosh(l_1 s^{\alpha})\cdots\cosh(l_{n-1} s^{\alpha})\sinh(l_n s^{\alpha})}\right\}\|_{L^{1}(0,t)} \\
		 	&\leq
		 	2^{n}\left[1+\frac{L(t)}{2l_1}\right]\cdots\left[1+\frac{L(t)}{2l_n}\right]\exp\left(-n\Lambda\left(\frac{l_2+\cdots+l_n-2l}{nt^{\nu}}\right)^{\beta}\right).
		 	%&\left.\kern-\nulldelimiterspace \hspace{150pt}+\exp\left(-n\Lambda\left(\frac{l_2+\cdots+l_n+2l_1-2l}{nt^{\nu}}\right)^{\beta}\right)\right].
		 \end{align*}
	\end{enumerate}}%
Where $L(t) = t^{\alpha}\Gamma(2-\alpha) / \Lambda^{(1-\alpha)}, \beta = 1/(1-\alpha)$.
\end{lemma}
\begin{proof} All five estimates follow the same procedure. But according to the combinations of the hyperbolic functions the results may vary, so we particularly show each decomposition upto some known level. To avoid inconvenience we only use $\|.\|$ instead of $\|.\|_{L^{1}(0,t)}$.  
	\begin{enumerate}
		\item [(i)]
		Let
		\begin{align*}
			\hat{Q}_1(s) := \frac{\exp(2l s^{\alpha})\cosh((l_1-l_2)s^{\alpha})}{\cosh(l_1 s^{\alpha})\cosh(l_2 s^{\alpha})} = \frac{\exp(-2(l_1-l)s^{\alpha})+ \exp(-2(l_2-l)s^{\alpha})}{(1+\exp(-2l_1 s^{\alpha}))(1+\exp(-2l_2 s^{\alpha}))}.
		\end{align*}
	   Therefore taking $L^1$ norm and using the triangular inequality give:
		{\small\begin{align}\label{res1}
			&\|\mathcal{L}^{-1}\{\hat{Q}_1(s)\}\| \nonumber\\
			& \leq  2\left\|\mathcal{L}^{-1}\left\{\frac{\exp(-2(l_1-l)s^{\alpha})}{(1+\exp(-2l_1 s^{\alpha}))(1+\exp(-2l_2 s^{\alpha}))}\right\}\right\| \\
			& \hspace{50pt}+2\left\|\mathcal{L}^{-1}\left\{\frac{\exp(-2(l_2-l)s^{\alpha})}{(1+\exp(-2l_1 s^{\alpha}))(1+\exp(-2l_2 s^{\alpha}))}\right\}\right\|. \nonumber
		\end{align}}%
	Now using the part (i) of Lemma \ref{SimpleLaplaceLemma} on each term of \eqref{res1} we have:
	    {\small\begin{align*}
	        &\|\mathcal{L}^{-1}\{\hat{Q}_1(s)\}\| \\
			& \leq  2\left\|\mathcal{L}^{-1}\left\{\frac{\exp(-(l_1-l)s^{\alpha})}{(1+\exp(-2l_1 s^{\alpha}))}\right\}\right \|\left\|\mathcal{L}^{-1}\left\{\frac{\exp(-(l_1-l)s^{\alpha})}{(1+\exp(-2l_2 s^{\alpha}))}\right\}\right\| \\ &\hspace{50pt}+2\left\|\mathcal{L}^{-1}\left\{\frac{\exp(-(l_2-l)s^{\alpha})}{(1+\exp(-2l_1 s^{\alpha}))}\right\}\right\| \left\|\mathcal{L}^{-1}\left\{\frac{\exp(-(l_2-l)s^{\alpha})}{(1+\exp(-2l_2 s^{\alpha}))}\right\}\right\|.
		\end{align*}}%
	Finally applying Lemma \ref{estimate_1} we have our result.
		\item [(ii)] The proof is obvious using part (i), hence omitted.
		\item [(iii)] We define
		{\small\begin{align*}
			&\hat{Q}_2(s) 
			:= \frac{\exp(2l s^{\alpha})\sinh(l_1 s^{\alpha})\sinh(l_n s^{\alpha})}{\cosh(l_1 s^{\alpha})\cosh(l_2 s^{\alpha})\cdots \cosh(l_n s^{\alpha})} \\
			&= -\frac{\exp(2l s^{\alpha})}{\cosh(l_2 s^{\alpha})\cdots \cosh(l_{n-1} s^{\alpha})}\left(1-\frac{\sinh(l_1 s^{\alpha})\sinh(l_n s^{\alpha})}{\cosh(l_1 s^{\alpha})\cosh(l_n s^{\alpha})}-1\right).
		\end{align*}}%
	Hence 
		{\small\begin{align*}
			\left\|\mathcal{L}^{-1}\{\hat{Q}_2(s)\}\right\| 
			& \leq \left\|\mathcal{L}^{-1}\left\{\frac{\exp(2l s^{\alpha})}{\cosh(l_2 s^{\alpha})\cdots \cosh(l_{n-1} s^{\alpha})}\right\}\right\|\\
			& \hspace{30pt}+ \left\|\mathcal{L}^{-1}\left\{\frac{\exp(2l s^{\alpha})\cosh((l_1-l_n) s^{\alpha})}{\cosh(l_1 s^{\alpha})\cosh(l_2 s^{\alpha})\cdots \cosh(l_n s^{\alpha})}\right\}\right\|
			 =: I_1 + I_2.
		\end{align*}}%
	Now we consider $I_1$ and $I_2$ separately to  obtain the estimates using the part (i) of Lemma \ref{SimpleLaplaceLemma}:
		{\small\begin{align*}
			I_1 
			&= 2^{n-2}\left\|\mathcal{L}^{-1}\left\{\frac{\exp(-(l_2+\cdots+l_{n-1}-2l)s^{\alpha})}{(1+\exp(-2l_2s^{\alpha}))\cdots (1+\exp(-2l_{n-1}s^{\alpha}))}\right\}\right\|\\
			& \leq  2^{n-2}\left\|\mathcal{L}^{-1}\left\{\frac{\exp(-\frac{l_2+\cdots+l_{n-1}-2l}{n-2}s^{\alpha})}{(1+\exp(-2l_2s^{\alpha}))}\right\}\right\|\cdots\left\|\mathcal{L}^{-1}\left\{\frac{\exp(-\frac{l_2+\cdots+l_{n-1}-2l}{n-2}s^{\alpha})}{(1+\exp(-2l_{n-1}s^{\alpha}))}\right\}\right\|:=J_1.
		\end{align*}}%
		And we have
		{\small\begin{align*}
			I_2 &\leq 2J_1\left\|\mathcal{L}^{-1}\left\{\frac{\exp(-2(l_1-l)s^{\alpha})+ \exp(-2(l_n-l)s^{\alpha})}{(1+\exp(-2l_1 s^{\alpha}))(1+\exp(-2l_n s^{\alpha}))}\right\}\right\|.
		\end{align*}}%
	Finally applying Lemma \ref{estimate_1} on both $I_1$ and $I_2$, we have our result.
		\item[(iv)] The procedure of the proof is similar to part (iii), hence omitted.
		\item[(v)] Define
		{\small\begin{align*}
			\hat{Q}_3(s) &:=\frac{\exp(2l s^{\alpha})\sinh(l_1 s^{\alpha})}{\cosh(l_1 s^{\alpha})\cdots\cosh(l_{n-1} s^{\alpha})\sinh(l_n s^{\alpha})} \\
			& = \frac{\exp(2l s^{\alpha})}{\cosh(l_2 s^{\alpha})\cdots\cosh(l_{n-1} s^{\alpha})\sinh(l_n s^{\alpha})(1+\exp(-2l_1 s^{\alpha}))}\\ 
			&\hspace{50pt}- \frac{\exp(-2(l_1-l) s^{\alpha})}{\cosh(l_2 s^{\alpha})\cdots\cosh(l_{n-1} s^{\alpha})\sinh(l_n s^{\alpha})(1+\exp(-2l_1 s^{\alpha}))}\\
			& =:\hat{I}_3 - \hat{I}_4,
		\end{align*}}%
	Separately considering $\hat{I}_3$ and $\hat{I}_4$ we have:
		{\small\begin{align*}
			&\|\mathcal{L}^{-1}\{\hat{I}_3\}\| \\
			 &\leq 2^{n-1}\left\|\mathcal{L}^{-1}\left\{\frac{\exp(-(l_2+\cdots+l_{n}-2l)s^{\alpha})}{(1+\exp(-2l_1s^{\alpha}))\cdots (1+\exp(-2l_{n-1}s^{\alpha}))(1-\exp(-2l_{n}s^{\alpha}))}\right\}\right\| \\
			& \leq  2^{n-1}\left\|\mathcal{L}^{-1}\left\{\frac{\exp(-\frac{l_2+\cdots+l_{n}-2l}{n}s^{\alpha})}{(1+\exp(-2l_1s^{\alpha}))}\right\}\right\|\cdots\left\|\mathcal{L}^{-1}\left\{\frac{\exp(-\frac{l_2+\cdots+l_{n}-2l}{n}s^{\alpha})}{(1+\exp(-2l_{n-1}s^{\alpha}))}\right\}\right\|\\
			&\hspace{205pt}\left\|\mathcal{L}^{-1}\left\{\frac{\exp(-\frac{l_2+\cdots+l_{n}-2l}{n}s^{\alpha})}{(1-\exp(-2l_{n}s^{\alpha}))}\right\}\right\|,
		\end{align*}}%
		and
		{\small\begin{align*}
			&\|\mathcal{L}^{-1}\{\hat{I}_4\}\| \\
			 &\leq  2^{n-1}\left\|\mathcal{L}^{-1}\left\{\frac{\exp(-\frac{l_2+\cdots+l_{n}+2l_1-2l}{n}s^{\alpha})}{(1+\exp(-2l_1s^{\alpha}))}\right\}\right\|\cdots\left\|\mathcal{L}^{-1}\left\{\frac{\exp(-\frac{l_2+\cdots+l_{n}+2l_1-2l}{n}s^{\alpha})}{(1+\exp(-2l_{n-1}s^{\alpha}))}\right\}\right\|\\
			 &\hspace{190pt}\left\|\mathcal{L}^{-1}\left\{\frac{\exp(-\frac{l_2+\cdots+l_{n}+2l_1-2l}{n}s^{\alpha})}{(1-\exp(-2l_{n}s^{\alpha}))}\right\}\right\|.
		\end{align*}}%
		Finally applying Lemma \ref{estimate_1} on both $\hat{I}_3$ and $\hat{I}_4$, we have our result.
	\end{enumerate}	
\hfill\end{proof}

\begin{theorem}[Convergence of DNWR for $1/2 < \nu < 1$]
		For fixed $T>0$, $0 < \nu \leq 1/2$ and $\theta_i = 1/(1+\sqrt{\kappa_{i+1}/\kappa_i}), 1 \leq i \leq m$ and $\theta_i = 1/(1+\sqrt{\kappa_{i}/\kappa_{i+1}}), m+1 \leq i \leq 2m$ , the DNWR algorithm \eqref{DNWR1}-\eqref{DNWR4} for $2m+1( m \in \mathbb{N})$ subdomains converge super-linearly with the estimate:
	\begin{align*}
		&\max_{1\leq i \leq 2m} \|w_i^{(k)}\|_{L^{\infty}(0,T)} \leq d^k \exp(-A k^{1/(1-\nu)})\max_{1\leq j \leq 2m} \|w_j^{(0)}\|_{L^{\infty}(0,T)},
	\end{align*} 
	where,  $h_j = a_j/\sqrt{\kappa_j}$, $h = \min_{1\leq j \leq 2m+1}h_j/2$, $ A = (1-\nu)\nu^{\nu/(1-\nu)}(2h/T^{\nu})^{1/(1-\nu)}$, $d= \|D\|_{\infty}$, and $\|.\|_{\infty}$ imply the infinity norm of matrix $D$, given in \eqref{eq2}. 
\end{theorem}
\begin{proof}
	The idea to proceed the proof for diffusion-wave case is similar to the sub-diffusion one but the estimation technique is different due to the lack of positivity results, which is one of the key ingredients in sub-diffusion.  
	The recursive equation in terms of $k^{th}$ and $(k-1)^{th}$ iteration on the error of artificial boundary takes the form from \eqref{proof9}:
	\begin{equation*}
		\hat{w}_i^{(k)}(s) = \sum_{j=1}^{2m}\hat{\rho}_{i,j} \hat{w}_j^{(k-1)}(s), \quad 1 \leq i, j \leq 2m.
	\end{equation*}
	 Define $\zeta := \exp(2hs^{\nu})$ and $\hat{\chi}^{(k)}(s) := \zeta^k\hat{w}^{(k)}(s)$. Therefore, the error equation converts to:
	 \begin{equation*}
	 	\hat{\chi}_i^{(k)}(s) = \sum_{j=1}^{2m}\zeta\hat{\rho}_{i,j} \hat{\chi}_j^{(k-1)}(s), \quad 1 \leq i, j \leq 2m.
	 \end{equation*}
    Taking the inverse Laplace transform on both side and then sup-norm on time reduce the equation:
    \begin{align*}
    	\left\|\chi_i^{(k)}\right\|_{L^{\infty}(0,T)} 
    	&\leq \sum_{j=1}^{2m} \left\|\mathcal{L}^{-1}\{\zeta\hat{\rho}_{i,j}\}\ast \chi_j^{(k-1)}\right\|_{L^{\infty}(0,T)}.
    \end{align*}
     Using part (ii) of Lemma \ref{SimpleLaplaceLemma} and taking maximum over $j$ give us:
     \begin{align} \label{eq1}
     	\left\|\chi_i^{(k)}\right\|_{L^{\infty}(0,T)} 
     	 \leq \max_{1\leq j \leq 2m} \left\|\chi_j^{(k-1)}\right\|_{L^{\infty}(0,T)} \sum_{j=1}^{2m} \left\|\mathcal{L}^{-1}\{\zeta\hat{\rho}_{i,j}\}\right\|_{L^{1}(0,T)}.
     \end{align}
     Now taking the $L^1$ norm on each term on $\mathcal{L}^{-1}\{\zeta\hat{\rho}_{i,j}\}$ give us the estimate matrix $D= d_{i,j}$, which follow directly from Lemma \ref{lemma_1}. Therefore
     for $1 \leq i \leq m, 1 \leq j \leq 2m$
     {\small\begin{align} \label{eq2}
     	\bar{d}_{i,i} &= 2\theta^*_i\sqrt{\frac{\kappa_{i+1}}{\kappa_{i}}}\left[1+\frac{L_T}{2h_i}\right]\left[1+\frac{L_T}{2h_{i+1}}\right]\left[\exp\left(-2B\left(\frac{h_{i}-h}{T^{\nu}}\right)^{\nu_1}\right)\right.\\
     	&\hspace{225pt}\left.\kern-\nulldelimiterspace+\exp\left(-2B\left(\frac{h_{i+1}-h}{T^{\nu}}\right)^{\nu_1}\right)\right]  \nonumber \\
     	&\leq 4\theta^*_i\sqrt{\frac{\kappa_{i+1}}{\kappa_{i}}} \left[1+\frac{L_T}{4h}\right]^2 \exp\left(-2B\left(\frac{h}{T^{\nu}}\right)^{\nu_1}\right) := d_{i,j}, \nonumber
     \end{align}}%
      {\small\begin{align*}
     	\bar{d}_{i,j} &= 2\theta^*_i\left[1+\frac{L_T}{2h_i}\right]\exp\left(-2B\left(\frac{h_{i}-2h}{T^{\nu}}\right)^{\nu_1}\right) \leq 2\theta^*_i\left[1+\frac{L_T}{4h}\right]:= d_{i,j}, \hspace{10pt}\mbox{for $j = i-1, i>1$}
     	 \end{align*}}%
      for $i+1 \leq j \leq m$
      {\small\begin{align*}
     	\bar{d}_{i,j} &= 2^{j-i}\theta^*_i\sqrt{\frac{\kappa_{j+1}}{\kappa_{i}}}\left[1+\frac{L_T}{2h_{i+1}}\right]\cdots\left[1+\frac{L_T}{2h_{j}}\right]\exp\left(-(j-i)B\left(\frac{h_{i+1}+\cdots+h_j-2h}{(j-i)T^{\nu}}\right)^{\nu_1}\right)\\
     	&\left[1+2\left[1+\frac{L_T}{2h_i}\right]\left[1+\frac{L_T}{2h_{j+1}}\right]\left[\exp\left(-2B\left(\frac{h_{i}-h}{T^{\nu}}\right)^{\nu_1}\right)+\exp\left(-2B\left(\frac{h_{j+1}-h}{T^{\nu}}\right)^{\nu_1}\right)\right]\right] \\
     	&\leq 2^{j-i}\theta^*_i\sqrt{\frac{\kappa_{j+1}}{\kappa_{i}}}\left[1+\frac{L_T}{4h}\right]^{j-i}\exp\left(-(j-i)B\left(\frac{2(j-i-1)h}{(j-i)T^{\nu}}\right)^{\nu_1}\right)\\
     	&\hspace{160pt}\left[1+4\left[1+\frac{L_T}{4h}\right]^2\exp\left(-2B\left(\frac{h}{T^{\nu}}\right)^{\nu_1}\right)\right]:= d_{i,j},
     \end{align*}}%
       {\small\begin{align*}
     	\bar{d}_{i,m+1} 
     	&=2^{m-i+2}\theta^*_i\sqrt{\frac{\kappa_{m+1}}{\kappa_{i}}}\left[1+\frac{L_T}{2h_{i}}\right]\cdots\left[1+\frac{L_T}{2h_{m+1}}\right]\\
     	&\hspace{130pt}\exp\left(-(m-i+2)B\left(\frac{h_{i+1}+\cdots+h_{m+1}-2h}{(m-i+2)T^{\nu}}\right)^{\nu_1}\right)\\
     	&\leq 2^{m-i+2}\theta^*_i\sqrt{\frac{\kappa_{m+1}}{\kappa_{i}}}\left[1+\frac{L_T}{4h}\right]^{m-i+2}\exp\left(-(m-i+2)B\left(\frac{2(m-i+1)h}{(m-i+2)T^{\nu}}\right)^{\nu_1}\right):= d_{i,j},
     	\end{align*}}%	
     	{\small\begin{align*}
     	d_{i,j} =0  \hspace{290pt}\mbox{otherwise},
     \end{align*}}%
     where $B = (1-\nu)\nu^{\nu/(1-\nu)},L_T = T^{\nu}\Gamma(2-\nu) / B^{(1-\nu)}, \nu_1 = 1/(1-\nu)$. Again the symmetric structure helps us to find $d_{i,j}$ for $m+1 \leq i \leq 2m$ for all $j$ and we achieve it by replacing all indices $i_1$ by $2m+2-i_1$, except $d_{i,j}, \theta^*_i$, where $i_1$ is replaced by $2m+1-i_1$. Now we define infinity norm on matrix $D$ by $d:= \|D\|_{\infty}$. Therefore taking maximum over $i$ in \eqref{eq1} give:
      \begin{align}  \label{eq3}
     	\max_{1\leq i \leq 2m}\left\|\chi_i^{(k)}\right\|_{L^{\infty}(0,T)} 
     	\leq d\max_{1\leq j \leq 2m} \left\|\chi_j^{(k-1)}\right\|_{L^{\infty}(0,T)} .
     \end{align}
 Now using the same reduction technique from $\chi_i$ to $w_i$ as sub-diffusion on \eqref{eq3} we end up with the estimate:
 \begin{align*}
 	&\max_{1\leq i \leq 2m} \|w_i^{(k)}\|_{L^{\infty}(0,T)} \leq d^k \exp(-A k^{1/(1-\nu)})\max_{1\leq j \leq 2m} \|w_j^{(0)}\|_{L^{\infty}(0,T)}.
 \end{align*} 
Hence the theorem.
\hfill\end{proof}

\begin{remark}
	The estimates for even number of subdomains that is $2m, m \in \mathbb{N}$ is almost same as the odd case. Here first we construct recursive error equation \eqref{proof9} for $1 \leq i \leq m-1$ and then for $m \leq i \leq 2m-1$. Except $m = 1$, one can directly construct $\hat{\rho}_{i,j}$ matrix from general formula \eqref{dij}. 
\end{remark}

\section{Convergence of DNWR Algorithm in 2D}\label{sectionDNWR2D}
In $2D$, we consider the homogeneous variant, that is, diffusion coefficient $\kappa$ same throughout the domain, of our model problem \eqref{model_problem} in rectangular spatial domain $\Omega:= (0, L) \times (0,\pi)$. Zero Dirichlet boundary condition on $\partial\Omega$ is imposed, and subdivide $\Omega$ into $2m+1, m\in \mathbb{N}$ rectangular strips; size $ \Omega_{i} := (x_{i-1},x_i) \times (0,\pi), 1\leq i \leq 2m+1$. If we take $\psi_i(x,y,t)$ be the initial guess on artificial boundary $x=x_i$, then the error update on artificial boundary takes the form:
\begin{equation} \label{proof10}
	\hat{W}_i^{(k)}(s,n) = \sum_{j=1}^{2m}\hat{\rho}_{i,j}\left(\sqrt{s^{2\nu}+kn^2}\right) \hat{W}_j^{(k-1)}(s,n), \quad 1\leq i,j \leq 2m.
\end{equation}  
Here, $\hat{W}_i^{(k)}(s,n)$ represent first Fourier sine transform on $y$ axis and then Laplace transform on time and $\hat{\rho}_{i,j}$ takes the same form as in \eqref{dij} but definitions of $\sigma_{i}, \gamma_i$ change to $\sigma_{i} := \sinh(\frac{a_{i}}{\sqrt{\kappa_{i}}}\sqrt{s^{2\nu}+ \kappa n^2}), \gamma_{i} := \cosh(\frac{a_{i}}{\sqrt{\kappa_{i}}}\sqrt{s^{2\nu}+ \kappa n^2})$. Now we need some lemmas on estimates for further progression. 
% New Lemma 5
\begin{lemma}\label{estimate_2}
Let $g \in L^{\infty}(0,\pi)$, and $\tau >0$. Then for any $y \in \mathbb{R}$,
	{\small\begin{equation*}
	\frac{2}{\pi} \left|\int_{0}^{\pi} g(\eta) \sum_{n \geq 1} \exp(-n^2 \eta) \sin(n\eta)\sin(ny) d\eta \right| \leq \|g\|_{L^{\infty}(0,\pi)}.
	\end{equation*}}%
\end{lemma}

% New Lemma
\begin{lemma} \label{lemma_3}
    Let $K(n)$ represents the Fourier coefficient of $k(y)$ and 
	{\small\begin{equation*}
		\hat{K_1}(s,n) = \frac{\exp(-l_1{\sqrt{s^{2\alpha}+ \lambda n^2}})}{1+\exp(-l_2{\sqrt{s^{2\alpha}+ \lambda n^2}})}\hat{K_2}(s,n)
	\end{equation*}}%
    then for $l_1 < l_2, \lambda >0$ and $0<\alpha \leq 1/2$
	{\small\begin{equation*}
		\|k_1(.,.)\|_{L^{\infty}(0,t;L^{\infty}(0,\pi))} \leq \left[1+\frac{t^{\alpha}\Gamma(2-\alpha)}{l_2\Lambda^{(1-\alpha)}}\right] \exp\left(-\Lambda\left(\frac{l_1}{t^{\alpha}}\right)^{1/(1-\alpha)}\right) \|k_2(.,.)\|_{L^{\infty}(0,t;L^{\infty}(0,\pi))},
	\end{equation*}}%
	where $\Lambda = (1-\alpha)\alpha^{\alpha/(1-\alpha)}$.
\end{lemma}
\begin{proof}
	Our procedure to prove this lemma is straightforward; we first take the inverse Laplace transform and then add all the sine frequencies to achieve our estimate.
    Therefore first take the binomial expansion and then inverse Laplace transform we have:	
    {\small\begin{align*}
		K_1(t,n) = \mathcal{L}^{-1}\left\{\sum_{j = 0}^{\infty}(-1)^j\exp(-(l_1+jl_2)\sqrt{s^{2\alpha}+\lambda n^2})\right\}\ast K_2(t,n),
	   \end{align*}}%
   Using Efros theorem and part (i) of Lemma \ref{invL_expLemma} on each exponential term give:
       {\small\begin{align*}
		K_1(t,n) = \left[\sum_{j = 0}^{\infty}(-1)^n\int_0^{\infty} \exp(-\lambda n^2 \tau)F(\tau,t)d\tau\right]\ast K_2(t,n),
	\end{align*}}%
	where $F(\tau,t) := 2\alpha \tau t^{-(2\alpha +1)} M_{2\alpha}(\tau t^{-2\alpha})\frac{l_1+jl_2}{\sqrt{4\pi t^3}}\exp(\frac{-(l_1+jl_2)^2}{4\tau})$.
	Now adding all the Fourier sine terms with coefficients provide:
	{\small\begin{align*}
		k_1(t,y) &= \sum_{n \geq 1} K_1(t,n) \sin(ny) \\
		&= \sum_{n \geq 1} \int_{0}^{t}d\tau_1\sum_{j = 0}^{\infty}(-1)^n\int_0^{\infty}d\tau \exp(-\lambda n^2 \tau)F(\tau,\tau_1) K_2(t-\tau_1,n) \sin(ny) \\
		&= \sum_{n \geq 1} \int_{0}^{t}d\tau_1\sum_{j = 0}^{\infty}(-1)^n\int_0^{\infty}d\tau \exp(-\lambda n^2 \tau)F(\tau,\tau_1) \left[\frac{2}{\pi} \int_{0}^{\pi} d\eta k_2(t-\tau_1,\eta)\sin(n\eta)\right] \sin(ny) \\
		&=  \int_{0}^{t}d\tau_1\sum_{j = 0}^{\infty}(-1)^n\int_0^{\infty}d\tau F(\tau,\tau_1) \left[\frac{2}{\pi} \int_{0}^{\pi} d\eta k_2(t-\tau_1,\eta) \sum_{n \geq 1} \exp(-\lambda n^2 \tau)\sin(n\eta) \sin(ny)\right].  \\
	\end{align*}}%
   We have applied Fubini's theorem to exchange integration and summation. Further taking absolute value on both sides and apply Lemma \ref{estimate_2} give:
	{\small\begin{align*}
		&|k_1(t,y)|\\
		& \leq \int_{0}^{t}d\tau_1\sum_{j = 0}^{\infty}\int_0^{\infty}d\tau F(\tau,\tau_1) \left|\frac{2}{\pi} \int_{0}^{\pi} d\eta k_2(t-\tau_1,\eta) \sum_{n \geq 1} \exp(-\lambda n^2 \tau)\sin(n\eta) \sin(ny)\right|  \\
		& \leq \int_{0}^{t}d\tau_1\sum_{j = 0}^{\infty}\int_0^{\infty}d\tau F(\tau,\tau_1) \left\|k_2(t-\tau_1,.)\right\|_{L^{\infty}(0,\pi)}  \\
		& \leq \left\|k_2(.,.)\right\|_{L^{\infty}(0,t;L^{\infty}(0,\pi))} \int_{0}^{t}d\tau_1\sum_{j = 0}^{\infty}\int_0^{\infty}d\tau F(\tau,\tau_1).
	\end{align*}}%
	Again using Efros theorem and Lemma \ref{estimate_1}, we have our proof:
	{\small\begin{align*}
		\|k_1(.,.)\|_{L^{\infty}(0,t;L^{\infty}(0,\pi))} \leq \left[1+\frac{t^{\alpha}\Gamma(2-\alpha)}{l_2 \Lambda^{(1-\alpha)}}\right] \exp\left(-\Lambda\left(\frac{l_1}{t^{\alpha}}\right)^{1/(1-\alpha)}\right) \|k_2(.,.)\|_{L^{\infty}(0,t;L^{\infty}(0,\pi))}.
	\end{align*}}%
	\hfill\end{proof}

\begin{theorem}[Convergence of DNWR for sub-diffusion in $2D$ ]
	For fixed $T>0$, $0 < \nu \leq 1/2$ and $\theta_i = 1/(1+\sqrt{\kappa_{i+1}/\kappa_i}), 1 \leq i \leq m$ and $\theta_i = 1/(1+\sqrt{\kappa_{i}/\kappa_{i+1}}), m+1 \leq i \leq 2m$ , the DNWR algorithm in $2D$ for the update error equation \eqref{proof10} for $2m+1( m \in \mathbb{N})$ subdomains converge super-linearly with the estimate:
	\begin{align*}
		&\max_{1\leq i \leq 2m} \|w_i^{(k)}\|_{L^{\infty}(0,T;L^{\infty}(0,\pi))} \leq d^k \exp(-A k^{1/(1-\nu)})\max_{1\leq j \leq 2m} \|w_j^{(0)}\|_{L^{\infty}(0,T;L^{\infty}(0,\pi))},
	\end{align*} 
	where,  $h_j = a_j/\sqrt{\kappa_j}$, $h = \min_{1\leq j \leq 2m+1}h_j/2$, $ A = (1-\nu)\nu^{\nu/(1-\nu)}(2h/T)^{1/(1-\nu)}$, $d= \|D\|_{\infty}$, and $\|.\|_{\infty}$ imply the infinity norm of matrix $D$, given in \eqref{eq2}. 
\end{theorem}
\begin{proof}
	As one can see that the estimate for diffusion-wave case for $1D$ and sub-diffusion case for $2D$ is exactly same, so we only go through the proof briefly and only mention those terms which are different from $1D$ case. 
	
	 Define $\xi := \exp(2h\sqrt{s^{2\nu}+kn^2})$, where $h = \min_j h_j/2$. Setting $\hat{V}^{(k)}_i(s,n) := \xi^k \hat{W}^{(k)}_i(s,n)$ transform the error update equation \eqref{proof10} to:
	\begin{equation} \label{proof11}
		\hat{V}_i^{(k)}(s,n) = \sum_{j=1}^{n} \xi\hat{\rho}_{i,j}(\sqrt{s^{2\nu}+kn^2}) \hat{V}_j^{(k-1)}(s,n).
	\end{equation}
   Now first simplify the cosine and sine term to exponential as given in the proof of Lemma \ref{lemma_1} and then repeatedly applying Lemma \ref{lemma_3} on the simplify terms, we obtain:
    \begin{align} \label{eq4}
	\left\|v_i^{(k)}(.,.)\right\|_{L^{\infty}(0,T;L^{\infty}(0,\pi))} 
	& \leq d \max_{1\leq j \leq 2m} \left\|v_j^{(k-1)}(.,.)\right\|_{L^{\infty}(0,T;L^{\infty}(0,\pi))} \nonumber\\
	& \leq d^{k} \max_{1\leq j \leq 2m} \left\|w_j^{(0)}(.,.)\right\|_{L^{\infty}(0,T;L^{\infty}(0,\pi))},
   \end{align}
  where $d = \|D\|_{\infty}$. Now again applying the Lemma \ref{lemma_3} letting $l_2 \to \infty$ on $\hat{W}^{(k)}_i(s,n) := \xi^{-k} \hat{V}^{(k)}_i(s,n)$ give:
	\begin{equation} \label{eq5}
		\|w_i(.,.)\|_{L^{\infty}(0,t;L^{\infty}(0,\pi))} \leq  \exp\left(-Ak^{1/(1-\nu)}\right) \|v_i(.,.)\|_{L^{\infty}(0,t;L^{\infty}(0,\pi))}.
    \end{equation}
    Finally combining \eqref{eq4} \& \eqref{eq5} and taking maximum over $i$ we have our estimate: 
    \begin{align*}
    	&\max_{1\leq i \leq 2m} \|w_i^{(k)}\|_{L^{\infty}(0,T;L^{\infty}(0,\pi))} \leq d^k \exp(-A k^{1/(1-\nu)})\max_{1\leq j \leq 2m} \|w_j^{(0)}\|_{L^{\infty}(0,T;L^{\infty}(0,\pi))}.
    \end{align*} 
\hfill\end{proof}	
	
	\begin{remark}
	  We keep the $2D$ bounded multi-subdomain diffusion-wave proof for future work due to some estimation issue. If we take $y$-axis unbounded then the two sub-domain proof can be easily done. We leave that one because the procedure is same as the two sub-domain NNWR diffusion-wave proof, given in.
	\end{remark}

\appendix
\section{Necessary calculation for DNWR}\label{AppendixA}
To simplify the calculation we assign:
\begin{align}
	\hat{z}_{\mu}^{(k)}(s) &= \partial_{x} \hat{u}_{\mu+1}^{(k)}(x_{\mu},s) \textit{ for } 1 \leq {\mu} \leq m, \\
	\hat{z}_{\iota}^{(k)}(s) &= -\partial_{x} \hat{u}_{\iota}^{(k)}(x_{\iota},s) \textit{ for } m+1 \leq {\iota} \leq 2m.
\end{align}
Solving the BVP in Dirichlet and Neumann step in \eqref{DNWR1}-\eqref{DNWR2} we get
\begin{align}
	\hat{u}_{m+1}^{(k)}(x,s)&= \frac{1}{\sigma_{m+1}}\left(\hat{w}_{m+1}^{(k-1)}\sinh(x-x_m)\sqrt{s^{2\nu}/\kappa_{m+1}}+\hat{w}_{m}^{(k-1)}\sinh(x_{m+1}-x)\sqrt{s^{2\nu}/\kappa_{m+1}}\right), \label{DNWR5}\\
	\intertext{for $1 \leq \mu \leq m$}
	\hat{u}_{\mu}^{(k)}(x,s)&= \frac{1}{\gamma_{\mu}}\left(\frac{\kappa_{\mu+1}\hat{z}_{\mu}^{(k)}}{\sqrt{\kappa_{\mu}s^{2\nu}}}\sinh(x-x_{\mu-1})\sqrt{s^{2\nu}/\kappa_{\mu}}+\hat{w}_{\mu-1}^{(k-1)}\cosh(x_{\mu}-x)\sqrt{s^{2\nu}/\kappa_{\mu}}\right), \label{DNWR6}\\
	\intertext{for $m+2 \leq {\iota} \leq 2m+1$}
	\hat{u}_{\iota}^{(k)}(x,s)&= \frac{1}{\gamma_{\iota}}\left(\hat{w}_{\iota}^{(k-1)}\cosh(x-x_{\iota-1})\sqrt{s^{2\nu}/\kappa_{\iota}} + \frac{\kappa_{\iota-1}\hat{z}_{\iota-1}^{(k)}}{\sqrt{\kappa_{\iota}s^{2\nu}}}\sinh(x_{\iota}-x)\sqrt{s^{2\nu}/\kappa_{\iota}}\right), \label{DNWR7}
\end{align}
Substituting \eqref{DNWR6} into \eqref{DNWR3} and \eqref{DNWR7} into \eqref{DNWR4} and using the notation $\sigma_{\mu} := \sinh(\frac{a_{\mu}}{\sqrt{\kappa_{\mu}}}s^{\nu})$, $\gamma_{\mu} := \cosh(\frac{a_{\mu}}{\sqrt{\kappa_{\mu}}}s^{\nu})$ and define $\sigma_{\iota}, \gamma_{\iota}$ same way, the recurrence relation for $k\in \mathbb{N}$ becomes: 
\begin{align}
	\hat{w}_{\mu}^{(k)}(s)&= \frac{\theta_{\mu}}{\gamma_{\mu}}\left(\frac{\kappa_{\mu+1}\hat{z}_{\mu}^{(k)}}{\sqrt{\kappa_{\mu}s^{2\nu}}}\sigma_{\mu} + \hat{w}_{\mu-1}^{(k-1)}\right) +(1-\theta_{\mu})\hat{w}_{\mu}^{(k-1)}, \label{DNWR8}\\
	\hat{w}_{\iota}^{(k)}(s)&= \frac{\theta_{\iota}}{\gamma_{{\iota}+1}}\left(\hat{w}_{{\iota}+1}^{(k-1)} + \frac{\kappa_{\iota}\hat{z}_{{\iota}+1}^{(k)}}{\sqrt{\kappa_{{\iota}+1}s^{2\nu}}}\sigma_{{\iota}+1}\right) +(1-\theta_{\iota})\hat{w}_{\iota}^{(k-1)}, \label{DNWR9}
\end{align}
Where for $1 \leq \mu \leq m-1$
\begin{align}
	\hat{z}_{\mu}^{(k)} &= -\sqrt{\frac{s^{2\nu}}{\kappa_{\mu+1}}}\frac{\sigma_{\mu+1}}{\gamma_{\mu+1}}\hat{w}_{\mu}^{(k-1)} + \frac{\kappa_{\mu+2}}{\kappa_{\mu+1}\gamma_{\mu+1}}\hat{z}_{\mu+1}^{(k)}, \label{DNWR10}
\end{align}
for $m+2 \leq {\iota} \leq 2m+1$
\begin{align}
	\hat{z}_{\iota}^{(k)} &= \frac{\kappa_{{\iota}-1}}{\kappa_{\iota}\gamma_{\iota}}\hat{z}_{{\iota}-1}^{(k)} -\sqrt{\frac{s^{2\nu}}{\kappa_{\iota}}}\frac{\sigma_{\iota}}{\gamma_{\iota}}\hat{w}_{\iota}^{(k-1)}, \label{DNWR13}
\end{align}	
and
\begin{align}
	\hat{z}_m^{(k)} &= -\sqrt{\frac{s^{2\nu}}{\kappa_{m+1}}}\frac{\gamma_{m+1}}{\sigma_{m+1}}\hat{w}_m^{(k-1)} + \sqrt{\frac{s^{2\nu}}{\kappa_{m+1}}}\frac{1}{\sigma_{m+1}}\hat{w}_{m+1}^{(k-1)}, \label{DNWR11}\\
	\hat{z}_{m+1}^{(k)} &= \sqrt{\frac{s^{2\nu}}{\kappa_{m+1}}}\frac{1}{\sigma_{m+1}}\hat{w}_{m}^{(k-1)} -\sqrt{\frac{s^{2\nu}}{\kappa_{m+1}}}\frac{\gamma_{m+1}}{\sigma_{m+1}}\hat{w}_{m+1}^{(k-1)}. \label{DNWR12}
\end{align}

Choose: for $1 \leq \mu \leq m$,
$ \bar{w}_{\mu}^{(k)} = \gamma_{\mu} \hat{w}_{\mu}^k, \bar{z}_{\mu}^{(k)} = \frac{\sigma_{\mu}}{\sqrt{s^{2\nu}/\kappa_{\mu}}} \hat{z}_{\mu}^k$ and for $m+1 \leq {\iota} \leq 2m+1$,
$ \bar{w}_{\iota}^{(k)} = \gamma_{{\iota}+1} \hat{w}_{\iota}^k, \bar{z}_{\iota}^{(k)} = \frac{\sigma_{{\iota}+1}}{\sqrt{s^{2\nu}/\kappa_{{\iota}+1}}} \hat{z}_{\iota}^k$.
Obtain:
\begin{align}
	\intertext{for $1 \leq \mu \leq m$}
	\bar{w}_{\mu}^{(k)} &= \frac{\theta_{\mu}}{\gamma_{\mu-1}}\bar{w}_{\mu-1}^{(k-1)} +(1-\theta_\mu)\bar{w}_{\mu}^{(k-1)} + \theta_{\mu} \frac{\kappa_{\mu+1}}{\kappa_{\mu}}\bar{z}_{\mu}^{(k)}, \label{DNWR14}\\
	\intertext{for $m+1 \leq {\iota} \leq 2m$}
	\bar{w}_{\iota}^{(k)} &= \frac{\theta_{\iota}}{\gamma_{{\iota}+2}}\bar{w}_{{\iota}+1}^{(k-1)} +(1-\theta_{\iota})\bar{w}_{\iota}^{(k-1)} + \theta_{\iota} \frac{\kappa_{\iota}}{\kappa_{{\iota}+1}}\bar{z}_{\iota}^{(k)}, \label{DNWR15}
\end{align}
and
\begin{align}
	\intertext{for $1 \leq \mu \leq m-1$}
	\sqrt{\frac{\kappa_{\mu+1}}{\kappa_{\mu}}}\bar{z}_{\mu}^{(k)} &= -\frac{\sigma_{\mu} \sigma_{\mu+1}}{\gamma_{\mu}\gamma_{\mu+1}}\bar{w}_{\mu}^{(k-1)} + \frac{\kappa_{\mu+2}}{\kappa_{\mu+1}}\frac{\sigma_{\mu}}{\sigma_{\mu+1}\gamma_{\mu+1}}\bar{z}_{\mu+1}^{(k)}, \label{DNWR16}\\
	\sqrt{\frac{\kappa_{m+1}}{\kappa_{m}}}\bar{z}_m^{(k)} &= -\frac{\sigma_m \gamma_{m+1}}{\gamma_{m}\sigma_{m+1}}\bar{w}_m^{(k-1)} + \frac{\sigma_m}{\sigma_{m+1}\gamma_{m+1}}\bar{w}_{m+1}^{(k-1)}, \label{DNWR17}\\
	\sqrt{\frac{\kappa_{m+1}}{\kappa_{m+2}}}\bar{z}_{m+1}^{(k)} &= \frac{\sigma_{m+2}}{\gamma_{m}\sigma_{m+1}}\bar{w}_{m}^{(k-1)} - \frac{\gamma_{m+1}\sigma_{m+2}}{\sigma_{m+1}\gamma_{m+2}}\bar{w}_{m+1}^{(k-1)}, \label{DNWR18}\\
	\intertext{for $m+2 \leq {\iota} \leq 2m$}
	\sqrt{\frac{\kappa_{\iota}}{\kappa_{{\iota}+1}}}\bar{z}_{\iota}^{(k)} &= \frac{\kappa_{{\iota}-1}}{\kappa_{\iota}}\frac{\sigma_{{\iota}+1}}{\sigma_{\iota}\gamma_{\iota}}\bar{z}_{{\iota}-1}^{(k)} -\frac{\sigma_{\iota} \sigma_{{\iota}+1}}{\gamma_{\iota}\gamma_{{\iota}+1}}\bar{w}_{\iota}^{(k-1)}, \label{DNWR19}
\end{align}
Define: for $1 \leq \mu \leq m$ and $\left(\bar{\mathbf{w}}_L^{(k)}\right)_{\mu} = \bar{w}_{\mu}^{(k)}, \left(\bar{\mathbf{w}}_R^{(k)}\right)_{\mu} = \bar{w}_{m+\mu}^{(k)}, \left(\bar{\mathbf{z}}_L^{(k)}\right)_{\mu} = \bar{z}_{\mu}^{(k)}, \left(\bar{\mathbf{z}}_R^{(k)}\right)_{\mu} = \bar{z}_{\mu+m}^{(k)}$.
From \eqref{DNWR14} \eqref{DNWR15} 
\begin{align} \label{DNWR20}
	\bar{\mathbf{w}}_L{(k)} = T_L \bar{\mathbf{w}}_L^{(k-1)} + P_L \bar{\mathbf{z}}_L^{(k)}, \bar{\mathbf{w}}_R{(k)} = T_R \bar{\mathbf{w}}_R^{(k-1)} + P_R \bar{\mathbf{z}}_R^{(k)},
\end{align}
From \eqref{DNWR16} \eqref{DNWR19}
\begin{align} \label{DNWR21}
	U_L\bar{\mathbf{z}}_L{(k)} = D_L \bar{\mathbf{w}}_L^{(k-1)} + E_L \bar{\mathbf{w}}_R^{(k-1)}, 	U_R\bar{\mathbf{z}}_R{(k)} = D_R \bar{\mathbf{w}}_R^{(k-1)} + E_R \bar{\mathbf{w}}_L^{(k-1)},
\end{align}
\begin{equation*}
	T_L = 
	\begin{bmatrix}
		1-\theta_1\\
		\frac{\theta_2}{\gamma_1} & 1-\theta_2 \\
		\quad\quad \ddots & \quad\ddots \\
		& \frac{\theta_m}{\gamma_{m-1}}& 1-\theta_m\\      
	\end{bmatrix} \quad
	T_R =
	\begin{bmatrix}
		1-\theta_{m+1} & \frac{\theta_{m+1}}{\gamma_{m+3}}\\
		& 1-\theta_{m+2}	    & \frac{\theta_{m+2}}{\gamma_{m+4}}\\
		& \quad\ddots         & \quad\ddots \\
		&                &    \ddots     &\frac{\theta_{2m-1}}{\gamma_{2m+1}} \\
		&                &        & 1 - \theta_{2m}
	\end{bmatrix}
\end{equation*}
\begin{equation*}
	U_L = 
	\begin{bmatrix}
		\sqrt{\frac{\kappa_2}{\kappa_1}} & -\frac{\kappa_3}{\kappa_2}\frac{\sigma_1}{\sigma_2 \gamma_2}\\
		\quad\ddots & \quad\ddots \\
		&\ddots &  -\frac{\kappa_{m+1}}{\kappa_{m}}\frac{\sigma_{m-1}}{\sigma_{m} \gamma_m} \\
		&        &  \sqrt{\frac{\kappa_{m+1}}{\kappa_m}}   
	\end{bmatrix} \quad
	U_R =
	\begin{bmatrix}
		\sqrt{\frac{\kappa_{m+1}}{\kappa_{m+2}}}\\
		-\frac{\kappa_{m+1}}{\kappa_{m+2}}\frac{\sigma_{m+3}}{\sigma_{m+2} \gamma_{m+2}} & \ddots \\
		\quad\ddots &  \quad\ddots \\
		&-\frac{\kappa_{2m-1}}{\kappa_{2m}}\frac{\sigma_{2m+1}}{\sigma_{2m} \gamma_{2m}}  &  \sqrt{\frac{\kappa_{2m}}{\kappa_{2m+1}}}   
	\end{bmatrix}
\end{equation*}
and 
\begin{align*}
	D_L &= diag \left(-\frac{\sigma_1 \sigma_2}{\gamma_1 \gamma_2},\cdots,-\frac{\sigma_{m-1} \sigma_m}{\gamma_{m-1} \gamma_m},-\frac{\sigma_m \gamma_{m+1}}{\gamma_m \sigma_{m+1}}\right) \\
	D_R &= diag \left(-\frac{\gamma_{m+1} \sigma_{m+2}}{\sigma_{m+1} \gamma_{m+2}}-\frac{\sigma_{m+2} \sigma_{m+3}}{\gamma_{m+2} \gamma_{m+3}},\cdots,-\frac{\sigma_{2m} \sigma_{2m+1}}{\gamma_{2m} \gamma_{2m+1}},\right) \\
	P_L & =  diag \left(\theta_1\frac{\kappa_2}{\kappa_1},\cdots,\theta_m\frac{\kappa_{m+1}}{\kappa_m}\right), \\
	P_R & =  diag \left(\theta_{m+1}\frac{\kappa_{m+1}}{\kappa_{m+2}},\cdots,\theta_{2m}\frac{\kappa_{2m}}{\kappa_{2m+1}}\right), \\
\end{align*}
and
\begin{equation*}
	(E_L)_{m,1} = \frac{\sigma_m}{\sigma_{m+1} \gamma_{m+2}}, (E_R)_{1,m} = \frac{\sigma_{m+2}}{\gamma_{m}\sigma_{m+1} }, 
\end{equation*}
Using \eqref{DNWR20} \& \eqref{DNWR21}:
\begin{align}
	\bar{\mathbf{w}}_L^{(k)} &= (T_L + P_LU_L^{-1}D_L)\bar{\mathbf{w}}_L^{(k-1)} + P_LU_L^{-1}E_L\bar{\mathbf{w}}_R^{(k-1)}, \\
	\bar{\mathbf{w}}_R^{(k)} &= (T_R + P_RU_R^{-1}D_R)\bar{\mathbf{w}}_r^{(k-1)} + P_RU_R^{-1}E_R\bar{\mathbf{w}}_L^{(k-1)},
\end{align}
Now introduce $i,j$ for matrix indices, we have:
\begin{equation*}
	(P_LU_L^{-1})_{i,j} = 
	\begin{cases}
		\theta_i \sqrt{\frac{\kappa_{i+1}}{\kappa_i}}, i=j, \\
		\theta_i \sqrt{\frac{\kappa_{j+1}}{\kappa_i}} \frac{\sigma_i}{\gamma_{i+1}\cdots\gamma_j \sigma_j}, i<j, \\
		0, \textit{ otherwise }
	\end{cases}
	\quad	
	(P_RU_R^{-1})_{i,j} = 
	\begin{cases}
		\theta_{m+i} \sqrt{\frac{\kappa_{m+i+1}}{\kappa_{m+i}}}, i=j, \\
		\theta_{m+i} \sqrt{\frac{\kappa_{m+j+1}}{\kappa_{m+i}}} \frac{\sigma_{m+j}}{\sigma_{m+i}\gamma_{m+j+1}\cdots\gamma_{m+i}}, i>j, \\
		0, \textit{ otherwise }
	\end{cases}
\end{equation*}

\bibliographystyle{siam}
\bibliography{paper}
%\nocite{*}
\end{document}